\documentclass[reqno]{amsart}
\usepackage{amsthm,amsmath,amssymb}
\usepackage{stmaryrd,mathrsfs}
\usepackage[T1]{fontenc}
\usepackage{esint}
\usepackage{tikz}

\allowdisplaybreaks

\newcommand{\abs}[1]{|#1|}
\newcommand{\Babs}[1]{\Big|#1\Big|}

\newcommand{\Norm}[2]{\|#1\|_{#2}}
\newcommand{\bNorm}[2]{\big\|#1\big\|_{#2}}
\newcommand{\BNorm}[2]{\Big\|#1\Big\|_{#2}}
\newcommand{\pair}[2]{\langle #1,#2 \rangle}
\newcommand{\Bpair}[2]{\Big\langle #1,#2 \Big\rangle}
\newcommand{\ave}[1]{\langle #1\rangle}

\newcommand{\R}{\mathbb{R}}
\newcommand{\C}{\mathbb{C}}

\newcommand{\Z}{\mathbb{Z}}

\renewcommand{\P}[0]{\mathbb{P}}
\newcommand{\E}[0]{\mathbb{E}}
\newcommand{\D}[0]{\mathbb{D}}
\newcommand{\F}[0]{\mathbb{F}}
\newcommand{\eps}[0]{\varepsilon}

\swapnumbers \numberwithin{equation}{section}

\theoremstyle{plain}
\newtheorem{theorem}[equation]{Theorem}
\newtheorem{proposition}[equation]{Proposition}

\newtheorem{lemma}[equation]{Lemma}

\theoremstyle{definition}
\newtheorem{definition}[equation]{Definition}

\theoremstyle{remark}

\makeatletter
\@namedef{subjclassname@2020}{%
  \textup{2020} Mathematics Subject Classification}
\makeatother

\begin{document}

\title{Singular integrals in uniformly convex spaces}

\author{Tuomas Hyt\"onen}
\address{Department of Mathematics and Statistics, P.O.B.~68 (Pietari Kalmin katu 5), FI-00014 University of Helsinki, Finland}
\email{tuomas.hytonen@helsinki.fi}

%\date{\today}

\thanks{The author was supported by the Academy of Finland through project No. 346314 (Finnish Centre of Excellence in Randomness and Structures ``FiRST'')}
%----------classification, keywords, date

% 46E40(1973?now)Spaces of vector- and operator-valued functions
% 42B20(1980?now)Singular and oscillatory integrals (Calderón-Zygmund, etc.)
% 60G46(1980?now)Martingales and classical analysis

\maketitle

\begin{center}
This paper is dedicated to Professor Per H. Enflo.
\end{center}

\begin{abstract}
We consider the action of finitely truncated singular integral operators on functions taking values in a Banach space. Such operators are bounded for any Banach space, but we show a quantitative improvement over the trivial bound in any space with an equivalent uniformly convex norm. This answers a question asked by Naor and the author, who previously proved the result in the important special case of the finite Hilbert transforms.

The proof, which splits the operator into a cancellative part and two paraproducts, follows the broad outline of similar results for genuinely singular (non-truncated) operators in the narrower class of UMD spaces. Thus we revisit and survey the recent techniques behind such results, but our precise setting, the main theorem, and some aspects of its proof, are new.
 
Curiously, it turns out that the paraproducts admit somewhat better bounds than the full operator. In a large class of spaces other than UMD, they remain bounded even without the finite truncations.
\end{abstract}

\section{{Introduction}}

The theory of singular integrals of Banach space -valued functions is well developed in the setting of so-called UMD (unconditional martingale differences) spaces. In fact, this is known to be the maximal generality in which the boundedness of several basic singular operators, like the Hilbert transform, can be extended \cite{Bourgain:83,Burkholder:83}. 

In view of this, the title of this paper might seem outrageous. After all, the questions of boundedness of singular integrals depend only on the isomorphic structure of the target Banach space, and the class of spaces that admit an equivalent uniformly convex norm is known to be strictly larger than the class of UMD spaces \cite{Bourgain:83,Pisier:exemple}.

However, our point of departure is the observation that there are many singular integral type objects of finite nature, whose boundedness is qualitatively trivial but where one can still hope to beat the trivial estimates by more clever quantitative considerations. A prototype example consists of the finite Hilbert transforms
\begin{equation*}
  Hf(i)=\sum_{\substack{j=1 \\ j\neq i}}^N \frac{f(j)}{i-j},\quad i\in\{1,\ldots,N\}.
\end{equation*}
It is not difficult to check that $\Norm{Hf}{\ell_s^N(\mathcal{X})}\lesssim\log N\cdot\Norm{f}{\ell_s^N(\mathcal{X})}$ for every Banach space $\mathcal{X}$, while the uniform bound $\Norm{Hf}{\ell_s^N(\mathcal{X})}\lesssim\Norm{f}{\ell_s^N(\mathcal{X})}$ is equivalent to the boundedness of the usual Hilbert transform on $L_s(\R;\mathcal{X})$, and hence to the UMD property of $\mathcal{X}$. However, an intermediate behaviour $\Norm{Hf}{\ell_s^N(\mathcal{X})}\lesssim(\log N)^\theta\Norm{f}{\ell_s^N(\mathcal{X})}$ is also possible; this phenomenon was found, and its geometric implications explored, by Naor and the author \cite{HN:dicho}.

Motivated by this, in the present work, our aim is to obtain similar improvements over the trivial bound for a general class of singular integrals of finite type when acting on functions taking values in a space that admits a uniformly convex norm. Actually, we will not be make direct use of uniform convexity as such, but rather resort to the rich theory of its equivalent formulations. Recall that the spaces with an equivalent uniformly convex norm are precisely those with an equivalent uniformly smooth norm, and further the same as the super-reflexive  ones \cite{Enflo72}. They can further be given an equivalent $p$-uniformly smooth or $q$-uniformly convex norm \cite{Pisier:1975}, or, what is most relevant for the present needs, they satisfy related martingale inequalities known as martingale type and cotype \cite{Pisier:1975,Pisier:1986}:

\begin{definition}\label{def:martTypeCotype}
A Banach space $\mathcal{X}$ is said to have {\em martingale type} $p\in(1,2]$, if for some (equivalently, all) $s\in(1,\infty)$ and for every probability space (equivalently, every $\sigma$-finite measure space) $(E,\mathcal{E},\mu)$, every martingale $(f_n)_{n=0}^N$ of arbitrary finite length in $L_s(\mu;\mathcal{X})$ satisfies
\begin{equation*}
  \Norm{f_N}{L_s(\mu;\mathcal{X})}
  \lesssim\BNorm{\Big(\Norm{f_0}{\mathcal{X}}^p+\sum_{n=1}^N\Norm{f_n-f_{n-1}}{\mathcal{X}}^p\Big)^{1/p}}{L_s(\mu)},
\end{equation*}
where the implied constant may depend at most on $\mathcal{X}$, $p$, and $s$.

A Banach space $\mathcal{X}$ is said to have {\em martingale cotype} $q\in[2,\infty)$ if, for the same quantities as above, there holds
\begin{equation*}
  \BNorm{\Big(\Norm{f_0}{\mathcal{X}}^q+\sum_{n=1}^N\Norm{f_n-f_{n-1}}{\mathcal{X}}^q\Big)^{1/q}}{L_s(\mu)}
  \lesssim  \Norm{f_N}{L_s(\mu;\mathcal{X})}.
\end{equation*}
\end{definition}

In \cite[p. 221]{Pisier:1986}, the case $s=p$ is taken as the definition for martingale type $p$, and the case $s=q$ for martingale cotype $q$, but the equivalence with other values of $s$ is observed immediately afterwards.  See also \cite[Chapter 10]{Pisier:book} for more information on these and related conditions. The definition is usually formulated for probability spaces, but the equivalence with any $\sigma$-finite measure space follows easily (approximating arbitrary martingales by ones supported on a set of finite measure, and multiplying the measure by a constant to achieve a probability space). See \cite[Chapter 3]{HNVW1} for a systematic development of martingale theory in the $\sigma$-finite setting, and in particular \cite[Section 3.5.d]{HNVW1} for martingale type and cotype.

Representations of singular integrals with the help of martingales are known in various contexts, and this provides the link to our present aims.
Our main result can be formulated as follows. We refer the reader to Section \ref{sec:setup} for relevant definitions.

\begin{theorem}\label{thm:main}
Let $(E,d,\mu)$ be a doubling metric measure space, and
let $Tf(u)=\int_E K(u,v)f(v)d\mu(v)$ be an integral operator whose kernel $K$ satisfies the Calder\'on--Zygmund standard estimates and the additional finiteness property
\begin{equation*}
  K(u,v)=0\quad\text{unless}\quad r\leq d(u,v)<R.
\end{equation*}
Let $n:=1+\log(R/r)$.
\begin{enumerate}
  \item If $\mathcal{X}$ is any Banach space and $s\in[1,\infty]$, then $\Norm{Tf}{L_s(\mu;\mathcal{X})}\lesssim n\Norm{f}{L_s(\mu;\mathcal{X})}$.
  \item If $s\in(1,\infty)$ and $T$ is bounded on the scalar-valued $L_s(\mu)$ with $\Norm{Tf}{L_s(\mu)}\lesssim \Norm{f}{L_s(\mu)}$, and if $\mathcal{X}$ has an equivalent uniformly convex norm, then for some $\theta\in[0,1)$, we have $\Norm{Tf}{L_s(\mu;\mathcal{X})}\lesssim n^{\theta}\Norm{f}{L_s(\mu;\mathcal{X})}$.
  \item\label{it:theta} If $1<p\leq s\leq q<\infty$ and $\mathcal{X}$ has martingale type $p$ and martingale cotype $q$, then one can take $\theta=1/p-1/q$.
\end{enumerate}
\end{theorem}

Theorem \ref{thm:main} applies in particular to the case when $E=\{1,\ldots,N\}$, $d$ is the usual distance, $\mu$ is the counting measure, and $T$ is the finite Hilbert transform. Then one can take $r=1$ and $R=N$. This important case of Theorem \ref{thm:main} was obtained in \cite{HN:dicho}, where it is also shown that \eqref{it:theta} is the best possible statement in general. The question, answered by Theorem \ref{thm:main}, of extending this particular case to general Calder\'on--Zygmund operators was also raised in \cite{HN:dicho}. 

The proof of this special case used a transference between the finite Hilbert transform on $\{1,\ldots,N\}$ and truncations of the Hilbert transform on $\R$, as well as Petermichl's dyadic martingale representation of the latter \cite{Pet}. The possibility of working directly with the original operator, making use of the more general representation theorems of singular integrals developed in \cite{Hytonen:A2} and, in the context of abstract spaces, in \cite{NRV}, was observed there as an alternative. The present paper implements this programme in detail. By now, several variants of the dyadic representation theorem of \cite{Hytonen:A2} are available in the literature. The present approach is perhaps closest to the one designed in \cite{GH:2018}.

\section{{Set-up}}\label{sec:setup}

A metric measure space $(E,d,\mu)$ is said to be doubling $\mu$ is.a Borel measure on $E$ and the balls satisfy
\begin{equation*}
  \mu(B(u,2t))\lesssim\mu(B(u,t)).
\end{equation*}
We denote $V(u,t):=\mu(B(u,t))$ and $V(u,v):=V(u,d(u,v))$. It follows from doubling that $V(u,v)\asymp V(v,u)$.

Let $\dot E^2:=\{(u,v)\in E\times E:u\neq v\}$.
Let $\omega:[0,1]\to[0,\infty)$ be continuous, non-decreasing and doubling in the sense that $\omega(2t)\lesssim\omega(t)$.
A function $K:\dot E^2\to\C$ is called an $\omega$-standard kernel if
\begin{equation*}
  \abs{K(u,v)}\lesssim \frac{1}{V(u,v)}
\end{equation*}
for all $(u,v)\in\dot E^2$, and
\begin{equation*}
  \abs{K(u,v)-K(u,w)}+\abs{K(v,u)-K(w,u)}\lesssim\omega\Big(\frac{d(v,w)}{d(u,v)}\Big)\frac{1}{V(u,v)}
\end{equation*}
for all $(u,v)\in \dot E^2$ and $w\in E$ with $d(v,w)\leq\frac12d(u,v)$ (hence also $(u,w)\in\dot E^2$).

Some control of $\omega$ is usually required, and we define
\begin{equation*}
  \Norm{\omega}{\operatorname{Dini}^\nu}
  :=\int_0^1\omega(t)\Big(1+\log\frac1t\Big)^\nu\frac{dt}{t}.
\end{equation*}
Much of the classical theory of singular integrals is valid under the standard Dini condition with $\nu=0$, but slightly stronger conditions on this scale are also often required.
In Theorem \ref{thm:main} (and many other results in the area), it suffices to take $\nu=1$. In the course of the proof, we will see that somewhat less is actually enough, but we will not insist too much in this at this point. Many results in the literature are formulated for $\omega(t)=t^\delta$, which are easily seen to satisfy the Dini conditions with any $\nu\geq 0$.

In this paper, we will impose the following additional assumption: for some $0<r<R<\infty$,
\begin{equation}\label{eq:trunc}
  K(u,v)=0\quad\text{unless}\quad r\leq d(u,v)<R.
\end{equation}
This assumption effectively kills the singularity of the kernel $K(u,v)$, making the $L_s(\mu)$ boundedness of the related integral operator
\begin{equation*}
  Tf(s)=\int_E K(u,v)f(t)d\mu(t)
\end{equation*}
qualitatively trivial:

\begin{proposition}
Under the above assumptions,
\begin{equation*}
  \int_E \abs{K(u,v)}d\mu(u)+\int_E \abs{K(v,u)}d\mu(u)\lesssim 1+\log\frac{R}{r}=:n,
\end{equation*}
hence
\begin{equation}\label{eq:trivBd}
  \Norm{Tf}{L_s(\mu;\mathcal{X})}\lesssim n\cdot\Norm{f}{L_s(\mu;\mathcal{X})}\quad\text{for all}\quad s\in[1,\infty],\quad f\in L_s(\mu;\mathcal{X}),
\end{equation}
where $\mathcal{X}$ is an arbitrary Banach space.
\end{proposition}

\begin{proof}
\begin{equation*}
  \int_E \abs{K(u,v)}d\mu(u)
  \leq\sum_{j: r\leq 2^j\leq\frac12 R} \int_{2^j\leq\abs{u-v}<2^{j+1}}\abs{K(u,v)}d\mu(u),
\end{equation*}
where
\begin{equation*}
  \abs{K(u,v)}\leq\frac{1}{\lambda(u,d(u,v))}
  \lesssim\frac{1}{\lambda(v,d(u,v))}\leq \frac{1}{\lambda(v,2^j)},
\end{equation*}
thus
\begin{equation*}
  \int_{2^j\leq\abs{u-v}<2^{j+1}}\abs{K(u,v)}d\mu(u)
  \lesssim\frac{\mu(B(v,2^{j+1}))}{\lambda(v,2^j)}
  \leq\frac{\lambda(v,2^{j+1})}{\lambda(v,2^j)}\lesssim 1,
\end{equation*}
and hence
\begin{equation*}
  \int_E \abs{K(u,v)}d\mu(u)
  \lesssim\sum_{j: r\leq 2^j\leq\frac12 R} 1\lesssim 1+\log\frac{R}{r}.
\end{equation*}
The estimate of $K(v,u)$ follows by symmetry of the assumptions, and the estimate of $Tf$ is then standard.
\end{proof}

But the point we wish to address is conditions under which we can beat the trivial bound \eqref{eq:trivBd}.

The notation
\begin{equation*}
  n:=1+\log\frac{R}{r}
\end{equation*}
is motivated by the key example of $E=\{1,\ldots,2^n\}$. Since $1\leq d(u,v)<2^n$ for all $(u,v)\in\dot E^2$, the assumption \eqref{eq:trunc} is automatically satisfied with $r=1$, $R=2^n$, hence $1+\log R/r=1+n\log 2\asymp n$. Another basic example arises from usual (smooth) truncations of a standard kernel.

\section{{Dyadic cubes}}

A construction of dyadic cubes in general doubling metric spaces is first due to Christ \cite{Christ:90}. Rather than just one system, we will require an ensemble of dyadic systems equipped with a suitable probability measure, leading to a notion of a random dyadic system. In the generality of doubling metric (even quasi-metric) spaces, a first such construction is from \cite{HM:12} with subsequent variants in \cite{AH:13,HK:12,HT:14,NRV}. We will quote some results from \cite{AH:13}, specialising them to a metric space (thus taking $A_0=1$ for  the quasi-triangle constant featuring in \cite{AH:13}.)

There are reference points $x^k_\alpha$, where $k\in\Z$ and $\alpha\in\mathscr A_k$, some countable index sets. There is parameter space $\Omega=\prod_{k\in\Z}\Omega_k$, where each $\Omega_k$ is a copy of the same finite set $\Omega_0$. Thus, there is a natural probability measure on $\Omega$.

For every $\omega\in\Omega$, there is a partial order relation $\leq_\omega$ among pairs $(k,\alpha)$ and $(\ell,\gamma)$. Within a fixed level, we declare that $(k,\alpha)\leq_\omega(k,\beta)$ if and only if $\alpha=\beta$. The restriction of the relation $\leq_\omega$ between levels $k$ and $k+1$ depends only on the component $\omega_k$. The relation between arbitrary levels $\ell\geq k$ is obtained via extension by transitivity.

There are also new random reference points $z^k_\alpha=z^k_\alpha(\omega_k)$ such that $d(z^k_\alpha(\omega),x^k_\alpha)<2\delta^k$. (This follows from \cite[Eq. (2.1)]{AH:13} and the definition of $z^k_\alpha$ further down the same page.)

By \cite[Theorem 2.11]{AH:13}, there are open and closed sets $\tilde Q^k_\alpha(\omega)$ and $\bar Q^k_\alpha(\omega)$ that satisfy
\begin{equation*}
  B(z^k_\alpha,\tfrac16\delta^k)\subseteq\tilde Q^k_\alpha(\omega)\subseteq\bar Q^k_\alpha(\omega)\subseteq B(z^k_\alpha,6\delta^k).
\end{equation*}
We refer to these sets as ``(dyadic) cubes''. Here $\delta\in(0,1)$ is a small parameter that indicates the ratio of the scales of two consecutive generations of the dyadic system. In $\R^d$, one typically takes $\delta=\frac12$. These cubes  satisfy the disjointness
\begin{equation*}
  \tilde Q^k_\alpha(\omega)\cap\bar Q^k_\beta(\omega)=\varnothing\quad(\alpha\neq\beta)
\end{equation*}
and the covering properties
\begin{equation*}
  E=\bigcup_\alpha\bar Q^k_\alpha(\omega),\quad
  \bar Q^k_\alpha(\omega)=\bigcup_{\beta:(k+1,\beta)\leq_\omega(k,\alpha)}\bar Q^k_\beta(\omega).
\end{equation*}

While these properties are valid for each fixed $\omega\in\Omega$, we additionally have the following probabilistic properties:
\begin{equation}\label{eq:smallBdry}
\begin{split}
  \P_\omega\Big(u\in\bigcup_\alpha\ &\partial_\eps Q^k_\alpha(\omega)\Big)\leq C\eps^\eta,\\
  &\partial_\eps Q^k_\alpha(\omega):=\{u\in \bar Q^k_\alpha(\omega):d(u,(\tilde Q^k_\alpha(\omega))^c)<\eps\delta^k\},
\end{split}
\end{equation}
for some fixed $0<\eta\leq 1\leq C<\infty$ and all $\eps>0$, and in particular the negligible boundary:
\begin{equation*}
  \P_\omega\Big(x\in\bigcup_{k,\alpha} \partial Q^k_\alpha(\omega)\Big)\leq C\eps^\eta,\quad
  \partial Q^k_\alpha(\omega):=\bar Q^k_\alpha(\omega)\setminus \tilde Q^k_\alpha(\omega).
\end{equation*}
The sets $\tilde Q^k_\alpha(\omega)$ and $\bar Q^k_\alpha(\omega)$ (and hence also $\partial_\eps Q^k_\alpha(\omega)$ and $\partial Q^k_\alpha(\omega)$) depend only on $(\omega_i)_{i=k}^\infty.$ We will also need to understand the following event:

\begin{lemma}\label{lem:m0}
Let us fix $\eps>0$ so small that the bound in \eqref{eq:smallBdry} satisfies $C\eps^\eta\leq\frac12$.
Then there exists an $m_0\in\Z_+$ such that we have the following:
 For some $k\in\Z$ and $m\geq m_0$, let $x^{k+m}_\beta,x^{k+m}_\gamma$ be two reference points at level $k+m$ with $d(x^{k+m}_\beta,x^{k+m}_\gamma)\leq\frac12\eps\delta^k$. Consider the random event
\begin{equation*}
\begin{split}
  A &:=\{\omega\in\Omega: \exists\alpha\text{ such that }(k+m,\beta)\leq_\omega(k,\alpha)\text{ and }(k+m,\gamma)\leq_\omega(k,\alpha)\} \\
  &\phantom{:}=\{\omega\in\Omega:  \exists\alpha\text{ such that }\bar Q^{k+m}_\beta(\omega)\subseteq\bar Q^k_\alpha(\omega)\text{ and }
  \bar Q^{k+m}_\gamma(\omega)\subseteq\bar Q^k_\alpha(\omega)\}
\end{split}
\end{equation*}
Then
\begin{enumerate}
  \item\label{it:dep} $A$ depends only on the components $(\omega_{k+i})_{i=0}^{m-1}$, and
  \item $\P(A)\geq\frac12$,
\end{enumerate}
\end{lemma}

\begin{proof}
Claim \eqref{it:dep} is immediate from the facts that the restriction of the relation $\leq_\omega$ between levels $k$ and $k+1$ depends only on the component $\omega_{k}$, and that in general it is defined via extension by transitivity.

When $\eps>0$ is chosen as stated, with probability at least $\frac12$, the point $x^{k+m}_\beta$ is not contained in any $\partial_\eps Q^k_\alpha(\omega)$. In particular, choosing $\alpha$ so that $x^{k+m}_\beta\in \bar Q^k_\alpha(\omega)$ (which must exist by the covering property), it follows that $d(x^{k+m}_\beta,(\tilde Q^k_\alpha(\omega))^c)\geq \eps\delta^k$. Now suppose that $d(x^{k+m}_\gamma,x^{k+m}_\beta)\leq\frac12\eps\delta^k$. Then 
\begin{equation*}
\begin{split}
  d(z^{k+m}_\gamma(\omega),(\tilde Q^k_\alpha(\omega))^c) 
  &\geq d(x^{k+m}_\beta,(\tilde Q^k_\alpha(\omega))^c)-d(z^{k+m}_\gamma(\omega),x^{k+m}_\gamma) \\
  &\qquad -d(x^{k+m}_\gamma,x^{k+m}_\beta) \\
  &> \eps\delta^k-2\delta^{k+m}-\frac12\eps\delta^k
  \geq(\frac12\eps-\delta^{m_0})\delta^k>0
\end{split}
\end{equation*}
provided that $m_0$ is large enough.
Hence $z^{k+m}_\gamma(\omega)\in\tilde Q^k_\alpha(\omega)$, and thus $(k+m,\gamma)\leq_\omega(k,\alpha)$.
\end{proof}

In the sequel, we will drop the somewhat heavy notation above, and denote a typical $Q^k_\alpha(\omega)$ simply by $Q$. We let $\mathscr D_k=\mathscr D_k(\omega):=\{Q^k_\alpha(\omega):\alpha\in\mathscr A_k\}$ be the dyadic cubes of generation $k$. We also write $\ell(Q):=\delta^k$ if $Q\in\mathscr D_k$. We denote by $B_Q:=B(z^k_\alpha,6\delta^k)$ the (non-random!) ball that contains $Q=Q^k_\alpha(\omega)$.

\section{{A dyadic decomposition of finite kernels}}

With the dyadic cubes defined above, we have the corresponding averaging (or conditional expectation; hence the notation) operators
\begin{equation*}
  \ave{f}_Q:=\frac{1}{\mu(Q)}\int_Q f d\mu,\quad
  \E_Q f:=1_Q\ave{f}_Q,\quad \E_i f:=\sum_{Q\in\mathscr D_i}\E_Q f
\end{equation*}
and their differences
\begin{equation*}
  \D_i f:=\E_{i+1}f-\E_i f=\sum_{Q\in\mathscr D_i}\D_Q f,\quad
  \D_Q f=\sum_{P\in\operatorname{ch}(Q)}\E_P f-\E_Q f,
\end{equation*}
where, for $Q\in\mathscr D_i$, we denote by
\begin{equation*}
  \operatorname{ch}(Q):=\{P\in\mathscr D_{i+1}:P\subseteq Q\}
\end{equation*}
the collection of its children, which has some cardinality between $1$ and a fixed finite $M$ depending only on the geometry of $(E,d)$.

For any $-\infty<\sigma<\tau<\infty$, we can decompose
\begin{equation*}
  f=\E_{\sigma} f+\sum_{i=\sigma}^\tau\D_i f+\F_\tau f,\quad\F_\tau f:=(f-\E_\tau)f=\sum_{Q\in\mathscr D_\tau}1_Q(f-\ave{f}_Q)
\end{equation*}

\begin{lemma}\label{lem:tail}
If $\tau$ is large enough so that $\delta^\tau\ll r$, then
\begin{equation*}
  \Norm{T\F_\tau f}{L_s(\mu)}\lesssim \Norm{\omega}{\operatorname{Dini}}\Norm{f}{L_s(\mu)},\qquad
  \Norm{\omega}{\operatorname{Dini}}:=\int_0^1\omega(t)\frac{dt}{t}.
\end{equation*}
\end{lemma}

\begin{proof}
Since $\int_Q (f(v)-\ave{f}_Q)d\mu(v)=0$, we have
\begin{equation*}
\begin{split}
  T\F_\tau f(u)
  &=\sum_{Q\in\mathscr D_\tau}\int_Q K(u,v)(f(v)-\ave{f}_Q)d\mu(v) \\
  &=\sum_{Q\in\mathscr D_\tau}\int_Q (K(u,v)-K(u,z_Q))(f(v)-\ave{f}_Q)d\mu(v).
\end{split}
\end{equation*}
If $d(u,v)\leq r$ for all $v\in Q$ (thus also for $v=z_Q$), this $Q$ does not contribute to the sum above. If $d(u,v)>r$ for some $v\in Q$, then $d(u,z_Q)\geq d(u,v)-d(v,z_Q)\gtrsim r$ if $\ell(Q)=\delta^\tau \ll r$. Then
\begin{equation*}
\begin{split}
  \int_Q &\abs{(K(u,v)-K(u,z_Q))(f(v)-\ave{f}_Q)}d\mu(v) \\
  &\lesssim\int_Q\omega\Big(\frac{d(v,z_Q)}{d(u,z_Q)}\Big)\frac{1}{\lambda(u,d(u,z_Q))}\abs{f(v)-\ave{f}_Q}d\mu(v) \\
  &\lesssim\omega\Big(\frac{r}{d(u,z_Q)}\Big)\frac{1}{\lambda(u,d(u,z_Q))}\int_Q\abs{f(v)}d\mu(v) \\
  &\lesssim\int_Q\omega\Big(\frac{r}{d(u,v)+r}\Big)\frac{1}{\lambda(u,d(u,v)+r)}\abs{f(v)}d\mu(v) \\
  &=:\int_Q K_r(u,v)\abs{f(v)}d\mu(v)
\end{split}
\end{equation*}
Hence, summing over $Q\in\mathscr D_\tau$,
\begin{equation*}
  \abs{T\F_\tau f(u)}
  \lesssim\int_E  K_{r}(u,v)\abs{f(v)}d\mu(v).
\end{equation*}

The kernel here satisfies
\begin{equation*}
\begin{split}
  &\int_E K_r(u,v)d\mu(v) \\
  &\leq \int_{B(u,r)}\omega(1)\frac{1}{\lambda(u,r)}d\mu(v)
  +\sum_{j=0}^\infty \int_{2^j r\leq\abs{u-v}<2^{j+1}r}\omega(2^{-j})\frac{1}{\lambda(u,2^j r)}d\mu(v) \\
  &\leq \omega(1)\frac{1}{\lambda(u,r)}\mu(B(u,r))
  +\sum_{j=0}^\infty \omega(2^{-j})\frac{1}{\lambda(u,2^j r)} \mu(B(u,2^{j+1}r)) \\
  &\lesssim\sum_{j=0}^\infty\omega(2^{-j})\asymp\int_0^1\omega(t)\frac{dt}{t},
\end{split}
\end{equation*}
and a similar bound for $\int_E K_r(v,u)d\mu(v)$ by symmetry; hence
\begin{equation*}
  \BNorm{u\mapsto \int_E  K_{r}(u,v)\abs{f(v)}d\mu(v)}{L_s(\mu)}\lesssim \Norm{\omega}{\operatorname{Dini}}\Norm{f}{L_s(\mu)}
\end{equation*}
by standard considerations.
\end{proof}

\begin{lemma}\label{lem:tail2}
If $\sigma$ is small enough so that $\delta^\sigma\gtrsim R$, then
\begin{equation*}
  \Norm{T\E_{\sigma} f}{L_s(\mu;\mathcal{X})}
  \lesssim \sup_{P\in\mathscr D_{\sigma}}\frac{\Norm{T(1_P)}{L_s(\mu)}}{\mu(P)^{1/s}}\Norm{f}{L_s(\mu;\mathcal{X})}.
\end{equation*}
Thus $T\E_{\sigma}$ is bounded on $L_s(\mu;\mathcal{X})$ assuming the {\em cube testing condition} that the supremum above is finite.
\end{lemma}

\begin{proof}
If $v\in Q\in\mathscr D_{\sigma}$ and $u\in E$, then $K(u,v)$ is non-zero only if $d(u,v)<R\lesssim\ell(Q)$, and hence only if $u\in cB_Q$ for some constant $c$. Thus
\begin{equation*}
  T\E_{\sigma} f(u)
  =\sum_{Q\in\mathscr D_{\sigma}}\int_Q K(u,v)d\mu(v)\ave{f}_Q
  =\sum_{Q\in\mathscr D_{\sigma}} 1_{c B_Q}(u)T(1_Q)(u)\ave{f}_Q.
\end{equation*}
By the geometric doubling property, the balls $c B_Q$ have bounded overlap, and hence
\begin{equation*}
\begin{split}
  \Norm{T\E_{\sigma} f}{L_s(\mu)}
  &\lesssim\Big(\sum_{Q\in\mathscr D_{\sigma}} \Norm{1_{c B_Q}T(1_Q)\ave{f}_Q}{L_s(\mu;\mathcal{X})}^s\Big)^{1/s} \\
  &\leq\sup_{P\in\mathscr D_{\sigma}}\frac{\Norm{T(1_P)}{L_s(\mu)}}{\mu(P)^{1/s}}
  \Big(\sum_{Q\in\mathscr D_{\sigma}} \mu(Q) \Norm{\ave{f}_Q}{\mathcal{X}}^s\Big)^{1/s} \\
  &\leq\sup_{P\in\mathscr D_{\sigma}}\frac{\Norm{T(1_P)}{L_s(\mu)}}{\mu(P)^{1/s}}\Norm{f}{L_s(\mu;\mathcal{X})}.
\end{split}
\end{equation*}
\end{proof}

To prove the boundedness of $T$ on $L_s(\mu;\mathcal{X})$, we consider the pairing
\begin{equation*}
  \pair{Tf}{g},\quad f\in L_s(\mu;\mathcal{X}),\quad g\in L_{s'}(\mu;\mathcal{X}^\prime).
\end{equation*}
This can be expanded as
\begin{equation*}
  \pair{Tf}{g}
  =\pair{T\E_{\sigma} f}{g}+\sum_{i=\sigma}^\tau\pair{T\D_i f}{g}+\pair{T\F_\tau f}{g}.
\end{equation*}
We already dealt with the first and the last terms in Lemmas \ref{lem:tail} and \ref{lem:tail2}. In the sum, expanding also $g$, we obtain
\begin{equation}\label{eq:expand2}
\begin{split}
  \sum_{i=\sigma}^\tau\pair{T\D_i f}{g}
  &=\sum_{i=\sigma}^\tau\pair{T\D_i f}{\E_{\sigma} g}
  +\sum_{i,j=\sigma}^\tau\pair{T\D_i f}{\D_j g}+\sum_{i=\sigma}^\tau\pair{T\D_i f}{\F_\tau g} \\
\end{split}
\end{equation}
The double sum in \eqref{eq:expand2} can be reorganised as
\begin{equation*}
\begin{split}
  \sum_{i,j=\sigma}^\tau &\pair{T\D_i f}{\D_j g}
  =\sum_{\substack{i,j=a \\ i =j}}^b\pair{T\D_i f}{\D_j g}
   +\sum_{\substack{i,j=a \\ i < j}}^b\pair{T\D_i f}{\D_j g}
   +\sum_{\substack{i,j=a \\ i > j}}^b\pair{T\D_i f}{\D_j g} \\
  &=\sum_{i=a }^b\pair{T\D_i f}{\D_i g}
   +\sum_{j=\sigma}^\tau\pair{T(\E_j f-\E_{\sigma} f)}{\D_j g}
   +\sum_{i=\sigma}^\tau\pair{T\D_i f}{\E_i g-\E_{\sigma} g} \\
  &=\sum_{i=\sigma}^\tau\Big(\pair{T\D_i f}{\D_i g}+\pair{T\E_i f}{\D_i g}+\pair{T\D_i f}{\E_i g}\Big) \\
  &\qquad-\sum_{i=\sigma}^\tau\pair{T\E_{\sigma} f}{\D_j g}-\sum_{i=\sigma}^\tau\pair{T\D_i}{\E_{\sigma} g}.
\end{split}
\end{equation*}
Substituting back to \eqref{eq:expand2}, the last term right above cancels out with the first term in \eqref{eq:expand2}, leaving
\begin{equation*}
\begin{split}
  \sum_{i=\sigma}^\tau\pair{T\D_i f}{g}
  &=\sum_{i=\sigma}^\tau\Big(\pair{T\D_i f}{\D_i g}+\pair{T\E_i f}{\D_i g}+\pair{T\D_i f}{\E_i g}\Big) \\
  &\qquad-\sum_{i=\sigma}^\tau\pair{T\E_{\sigma} f}{\D_j g}+\sum_{i=\sigma}^\tau\pair{T\D_i f}{\F_\tau g}.
\end{split}
\end{equation*}
Here
\begin{equation*}
  \sum_{i=\sigma}^\tau\pair{T\E_{\sigma} f}{\D_j g}=\pair{T\E_{\sigma} f}{\E_\tau g-\E_{\sigma} g}
\end{equation*}
is controlled by Lemma \ref{lem:tail2} and the contractivity of $\E_i$, while
\begin{equation*}
  \sum_{i=\sigma}^\tau\pair{T\D_i f}{\F_\tau g}
  =\pair{\E_\tau-\E_{\sigma} f}{T^*\F_\tau g}
\end{equation*}
is controlled by an application of Lemma \ref{lem:tail} on the dual side, observing the symmetry of the assumptions.

Altogether, we find that
\begin{equation*}
\begin{split}
  &\Babs{\pair{Tf}{g}-\sum_{i=\sigma}^\tau\Big(\pair{T\D_i f}{\D_i g}+\pair{T\E_i f}{\D_i g}+\pair{T\D_i f}{\E_i g}\Big)} \\
  &\qquad\lesssim\Big(\sup_{P\in\mathscr D_{\sigma}}\frac{\Norm{T(1_P)}{L_s(\mu)}}{\mu(P)^{1/s}}+\Norm{\omega}{\operatorname{Dini}}\Big)\Norm{f}{L_s(\mu;\mathcal{X})}\Norm{g}{L_{s'}(\mu;\mathcal{X}^\prime)}.
\end{split}
\end{equation*}
This is valid for any Banach space $\mathcal{X}$, and hence the core is estimating the sum on the left. For the previous estimate, we need $\delta^\tau\ll r$ and $\delta^\sigma\gtrsim R$, where the implied constants are absolute. Thus we may take $\tau\asymp \log\frac1r$ and $\sigma\asymp\log\frac1R$ so that the length of the sum is
\begin{equation*}
  1+(\tau-\sigma)\asymp 1+\log\frac{R}{r}=n.
\end{equation*}
Proving a uniform bound for each term and applying the triangle inequality would recover the trivial estimate, which is linear in $n$, and which we already obtained by more elementary considerations. We wish to make use of the above decomposition to beat this trivial bound.

\section{{Estimates for Haar coefficients}}

The difference operators $\D_i$ have the piecewise local expansion
\begin{equation*}
  \D_i f=\sum_{Q\in\mathscr D_i}\D_Q f,
\end{equation*}
where each $\D_Q$, in turn, can be expanded in terms of a bounded number of ``Haar'' functions associated with $Q$; they are constant on the dyadic children of $Q$:
\begin{equation*}
  \D_Q f=\sum_{\alpha=1}^{m_Q-1}\pair{f}{h_Q^\alpha}h_Q^\alpha.
\end{equation*}
Here $m_Q\in\{1,\ldots,M\}$ is the number of dyadic children of $Q$. In general it can happen that $m_Q=1$, in which case the sum above is empty, and $\D_Q=0$. Denoting $h_Q^0:=\mu(Q)^{-1/2}1_Q$, we obtain similar formulas for the averaging operators
\begin{equation*}
  \E_i f=\sum_{Q\in\mathscr D_i}\E_Q f,\quad \E_Q f=\ave{f}_Q 1_Q=\pair{f}{h_Q^0}h_Q^0.
\end{equation*}
Thus each of the three types of terms $\pair{T\D_i f}{\D_i g}$, $\pair{T\E_i f}{\D_i g}$, $\pair{T\D_i f}{\E_i g}$ can be expanded in terms of a bounded number of sums of the type
\begin{equation*}
  \sum_{P,Q\in\mathscr D_i} \chi_{\alpha,\beta}\pair{Th_P^\alpha}{h_Q^\beta}\pair{f}{h_P^\alpha}\pair{g}{h_Q^\beta},
\end{equation*}
where $\chi_{\alpha,\beta}\in\{0,1\}$. Since there is at most one $\E_i$, at most one of $\alpha$ and $\beta$ can be $0$ in any given sum of this type.

Estimates for the Haar coefficients $\pair{Th_P^\alpha}{h_Q^\beta}$ are well known in the Euclidean setting. The extension to doubling metric spaces depends on the fact that a certain Hardy inequality remains valid for the dyadic cubes in this situation, as observed by Auscher and Routin \cite{AR:13}. We only need the following special case of their \cite[Lemma 2.4]{AR:13}: If $P,Q\in\mathscr D$ are disjoint, then
\begin{equation}\label{eq:Hardy}
  \int_P\int_Q\frac{1}{V(u,v)}d\mu(u)d\mu(v)\lesssim\mu(P)^{1/s}\mu(Q)^{1/s'}
\end{equation}
for $s\in(1,\infty)$. Then
\begin{equation*}
  \pair{Th_P^\alpha}{h_Q^\beta}
  =\sum_{\substack{ P'\in\operatorname{ch}(P) \\ Q'\in\operatorname{ch}(Q)}}\pair{T1_{P'}}{1_{Q'}}\ave{h_P^\alpha}_{P'}\ave{h_Q^\beta}_{Q'}.
\end{equation*}
If $\ell(P)=\ell(Q)$ and $P'\neq Q'$, then $P'\cap Q'=\varnothing$, and \eqref{eq:Hardy} applies to give
\begin{equation*}
\begin{split}
  \abs{\pair{T1_{P'}}{1_{Q'}}}
  &=\Babs{\int_{P'}\int_{Q'}K(u,v)dudv} \\
  &\lesssim\int_{P'}\int_{Q'}\frac{1}{V(u,v)}dudv
  \lesssim\mu(P')^{1/2}\mu(Q')^{1/2}.
\end{split}
\end{equation*}
Thus
\begin{equation}\label{eq:HaarClose}
\begin{split}
  &\Babs{\sum_{\substack{ P'\in\operatorname{ch}(P) \\ Q'\in\operatorname{ch}(Q) \\ P'\neq Q'}}\pair{T1_{P'}}{1_{Q'}}\ave{h_P^\alpha}_{P'}\ave{h_Q^\beta}_{Q'}} \\
  &\lesssim\sum_{P'\in\operatorname{ch}(P)} \mu(P')^{1/2} \abs{\ave{h_P^\alpha}_{P'}} \sum_{ Q'\in\operatorname{ch}(Q)} \mu(Q')^{1/2} \abs{\ave{h_Q^\beta}_{Q'}},
\end{split}
\end{equation}
where
\begin{equation*}
\begin{split}
  \sum_{P'\in\operatorname{ch}(P)} \mu(P')^{1/2} \abs{\ave{h_P^\alpha}_{P'}}
  &\leq M^{1/2}\Big(\sum_{P'\in\operatorname{ch}(P)} \mu(P') \abs{\ave{h_P^\alpha}_{P'}}^2\Big)^{1/2} \\
  &= M^{1/2}\Norm{h_P^\alpha}{L_2(\mu)}=M^{1/2},
\end{split}
\end{equation*}
and similarly for the sum over $Q'\in\operatorname{ch}(Q)$. Thus $\eqref{eq:HaarClose}\lesssim 1$.

If $P,Q\in\mathscr D_i$ are different, then all their children $P',Q'$ satisfy $P'\neq Q'$, and we have proved that
\begin{equation*}
  \abs{\pair{Th_P^\alpha}{h_Q^\beta}}\lesssim 1,\quad P,Q\in\mathscr D_i,\quad P\neq Q.
\end{equation*}
If $P=Q$, then we need to deal in addition with the sum
\begin{equation*}
\begin{split}
  &\Babs{\sum_{P'\in\operatorname{ch}(P) }\pair{T1_{P'}}{1_{P'}}\ave{h_P^\alpha}_{P'}\ave{h_P^\beta}_{P'}} \\
  &\qquad\leq\sup_{Q\in\mathscr D}\frac{\abs{\pair{T1_Q}{1_Q}}}{\mu(Q)}\sum_{P'\in\operatorname{ch}(P) }\mu(P')\abs{\ave{h_P^\alpha}_{P'}\ave{h_P^\beta}_{P'}},
\end{split}
\end{equation*}
where
\begin{equation*}
\begin{split}
  &\sum_{P'\in\operatorname{ch}(P) }\mu(P')\abs{\ave{h_P^\alpha}_{P'}\ave{h_P^\beta}_{P'}} \\
  &\qquad\leq\Big(\sum_{P'\in\operatorname{ch}(P) }\mu(P')\abs{\ave{h_P^\alpha}_{P'}}^2\Big)^{1/2}
  \Big(\sum_{P'\in\operatorname{ch}(P) }\mu(P')\abs{\ave{h_P^\beta}_{P'}}^2\Big)^{1/2} \\
  &\qquad=\Norm{h_P^\alpha}{L_2(\mu)}\Norm{h_P^\beta}{L_2(\mu)}=1.
\end{split}
\end{equation*}
The finiteness of
\begin{equation*}
  \Norm{T}{\operatorname{wbp}(\mathscr D)}:=\sup_{Q\in\mathscr D}\frac{\abs{\pair{T1_Q}{1_Q}}}{\mu(Q)}
\end{equation*}
is the so-called weak boundedness property. It follows from the finiteness of the testing quantity
\begin{equation*}
  \Norm{T}{\operatorname{test}^s(\mathscr D)}:=\sup_{Q\in\mathscr D}\frac{\Norm{T1_Q}{L_s(\mu)}}{\mu(Q)^{1/s}},
\end{equation*}
which in turn is clearly dominated by the operator norm of $T$ on $L_s(\mu)$.

If $P$ and $Q$ are well separated, we need and have a better estimate. Assuming for instance that $\alpha\neq 0$ (the case $\beta\neq 0$ being symmetric), we have
\begin{equation*}
\begin{split}
  \pair{Th_P^\alpha}{h_Q^\beta}
  &=\int_P\int_Q K(u,v)h_P^\alpha(v)h_Q^\beta(u)d\mu(u)d\mu(v) \\
  &=\int_P\int_Q [K(u,v)-K(u,z_P)]h_P^\alpha(v)h_Q^\beta(u)d\mu(u)d\mu(v).
\end{split}
\end{equation*}
If $d(P,Q)\gg\ell(P)=\ell(Q)$, then $v\in P$ and $u\in Q$ satisfy $d(v,z_P)\lesssim\ell(P)\ll d(u,v)$, and hence
\begin{equation}\label{eq:HaarFar}
\begin{split}
  \abs{\pair{Th_P^\alpha}{h_Q^\beta}}
  &\lesssim\int_P\int_Q \omega\Big(\frac{\ell(P)}{d(u,z_P)}\Big)\frac{1}{V(u,z_P)}\abs{h_P^\alpha(v)h_Q^\beta(u)}d\mu(u)d\mu(v) \\
  &\lesssim \omega\Big(\frac{\ell(P)}{d(z_Q,z_P)+\ell(P)}\Big)\frac{\Norm{h_P^\alpha}{L_1(\mu)}\Norm{h_Q^\beta}{L_1(\mu)}}{V(z_Q,d(z_Q,z_P)+\ell(P))} \\
  &\lesssim \omega\Big(\frac{\ell(P)}{d(z_Q,z_P)+\ell(P)}\Big)\frac{\sqrt{\mu(P)\mu(Q)}}{V(z_Q,d(z_Q,z_P)+\ell(P))}.
\end{split}
\end{equation}
Actually, this final bound is valid also for $d(P,Q)\lesssim\ell(P)=\ell(Q)$. In this case, $d(z_Q,z_P)+\ell(P)\asymp 1$, so that the argument of $\omega$ is roughly $1$, and $$V(z_Q,d(z_Q,z_P)+\ell(P))\asymp\mu(Q)\asymp\mu(P).$$
Thus, when $d(P,Q)\lesssim\ell(P)=\ell(Q)$, the bound \eqref{eq:HaarFar} reduces to the uniform bound obtained earlier.

\section{{Extracting paraproducts}}\label{sec:parap}

In the two sums
\begin{equation*}
  \sum_{i=\sigma}^\tau\pair{T\E_i f}{\D_i g},\quad  \sum_{i=\sigma}^\tau\pair{T\D_i f}{\E_i g},
\end{equation*}
we need to force some cancellation as follows. Let us deal with the first one, the second being symmetric. A typical term in this sum is
\begin{equation*}
  \pair{T\E_i f}{\D_i g}
  =\sum_{P,Q\in\mathscr D_i}\pair{T1_P}{\D_Q g}\ave{f}_P.
\end{equation*}
Writing
\begin{equation*}
  \ave{f}_P=(\ave{f}_P-\ave{f}_Q)+\ave{f}_Q,
\end{equation*}
we have introduced some cancellation into the first term. On the other hand, in the sum involving the second term, the only factor depending on $P$ is the indicator $1_P$ in the pairing $\pair{T1_P}{\D_Q g}$. Thus, summing over $P\in\mathscr D_i$ first, we obtain by linearity
\begin{equation*}
  \sum_{P,Q\in\mathscr D_i}\pair{T1_P}{\D_Q g}\ave{f}_Q
  =\sum_{Q\in\mathscr D_i}\Bpair{T\sum_{P\in\mathscr D_i}1_P}{\D_Q g}\ave{f}_Q
  =\sum_{Q\in\mathscr D_i}\pair{T1}{\D_Q g}\ave{f}_Q.
\end{equation*}
For the finite singular integrals that we consider, there is no trouble in making sense of the expression ``$T1$'' above, as the action of the integral operator $T$ on bounded function is well defined. Denoting (as usual) $b:=T1$, and using the self-adjointness of $\D_Q$, we find that
\begin{equation*}
%  \sum_{i=\sigma}^\tau \pair{T\E_i f}{\D_i g}=\sum_{i=\sigma}^\tau\sum_{P,Q\in\mathscr D_i} \pair{T1_P}{\D_Q g}(\ave{f}_P-\ave{f}_Q)
    \sum_{i=\sigma}^\tau \sum_{Q\in\mathscr D_i}\pair{b}{\D_Q g}\ave{f}_Q
    =\Bpair{\sum_{i=\sigma}^\tau \sum_{Q\in\mathscr D_i}\D_Q b\ave{f}_Q}{g}
    =:\Bpair{\Pi_b^{\sigma,\tau}f}{g}.
\end{equation*}
Here
\begin{equation*}
  \Pi_b^{\sigma,\tau}f
  :=\sum_{i=\sigma}^\tau \sum_{Q\in\mathscr D_i}\D_Q b\ave{f}_Q
  =\sum_{i=\sigma}^\tau \D_i b \E_i f
\end{equation*}
is a truncated version of a {\em dyadic paraproduct}. We can estimate it as follows. Noting that $\D_i$ are martingale differences, the assumption that $\mathcal{X}$ has martingale type $p\in[1,2]$ implies that
\begin{equation*}
\begin{split}
  \Norm{\Pi_b^{\sigma,\tau}f}{L_s(\mu;\mathcal{X})}
  &\lesssim\BNorm{\Big(\sum_{i=\sigma}^\tau\abs{\D_i b}^p\Norm{\E_i f}{\mathcal{X}}^p\Big)^{1/p}}{L_s(\mu;\mathcal{X})} \\
  &\leq n^{1/p-1/2}\BNorm{\Big(\sum_{i=\sigma}^\tau\abs{\D_i b}^2\Norm{\E_i f}{\mathcal{X}}^2\Big)^{1/2}}{L_s(\mu;\mathcal{X})} \\
  &\leq n^{1/p-1/2}\BNorm{\Big(\sum_{i=\sigma}^\tau\sum_{Q\in\mathscr D_i}\abs{\D_Q b}^2\ave{\Norm{f}{\mathcal{X}}}_Q^2\Big)^{1/2}}{L_s(\mu;\mathcal{X})}.
\end{split}
\end{equation*}

Next, we make a stopping time construction as follows. The initial layer of stopping cubes consist of all $Q\in\mathscr D_\sigma$. Assuming that some stopping cube $S\in\mathscr D$ has been picked, we look for its maximal dyadic subcubes $S'\subsetneq S$ such that either
\begin{equation*}
  \ave{\Norm{f}{\mathcal{X}}}_{S'}>4\ave{\Norm{f}{\mathcal{X}}}_S,
\end{equation*}
or
\begin{equation*}
  \sum_{\substack{Q\in \mathscr D \\ S'\subsetneq Q\subseteq S}}\abs{D_Q b(u)}^2>\lambda^2\quad\text{for all}\quad u\in S',
\end{equation*}
where $\lambda>0$ will be specified shortly. Note that the left-hand side of the stopping criterion is constant on $S'$, since $D_Q b$ is constant on the dyadic children of $Q$, so the condition ``for all $u\in S'$'' could be equivalently replaced by ``for some $u\in S'$''. Let us refer to the stopping cubes arising from the first or the second criterion as being of the first or the second kind, respectively.

For the stopping cubes of the first kind (pairwise disjoint by their maximality), it is immediate from the stopping criterion that
\begin{equation*}
  \sum_{S'}\mu(S')\leq\sum_{S'}\frac{1}{4\ave{\Norm{f}{\mathcal{X}}}_S}\int_{S'}\Norm{f}{\mathcal{X}}d\mu
  \leq\frac{1}{4\ave{\Norm{f}{\mathcal{X}}}_S}\int_S\Norm{f}{\mathcal{X}}d\mu=\frac{1}{4}\mu(S).
\end{equation*}

To derive a similar estimate for the stopping cubes of the second kind, let us consider the dyadic square function and its truncations
\begin{equation*}
  \mathfrak{S}b:=\Big(\sum_{Q\in\mathscr D}\abs{D_Q b}^2\Big)^{1/2},\quad
  \mathfrak{S}_P b:=\Big(\sum_{\substack{Q\in\mathscr D \\ Q\subseteq P}}\abs{D_Q b}^2\Big)^{1/2}
  =\mathfrak{S}(1_P(b-\ave{b}_P)).
\end{equation*}
If $S'\subsetneq S$ is one of the stopping cubes of the second kind, then
\begin{equation*}
  \mathfrak{S}_Sb(u)^2\geq \sum_{\substack{Q\in \mathscr D \\ S'\subsetneq Q\subseteq S}}\abs{D_Q b(u)}^2>\lambda^2\quad\text{for all}\quad u\in S'.
\end{equation*}
Thus
\begin{equation*}
\begin{split}
  \mu\Big(\bigcup S'\Big)\leq\mu(\mathfrak{S}_Sb>\lambda)
  &\leq\frac{1}{\lambda}^2\int_E (\mathfrak{S}_Sb)^2d\mu \\
  &=\frac{1}{\lambda^2}\int_S\abs{b-\ave{b}_S}^2d\mu 
  \leq\frac{1}{\lambda^2}\Norm{b}{\operatorname{BMO}^2(\mathscr D)}^2\mu(S),
\end{split}
\end{equation*}
where
\begin{equation*}
  \Norm{b}{\operatorname{BMO}^2(\mathscr D)}
  :=\sup_{Q\in\mathscr D}\Big(\frac{1}{\mu(Q)}\int_Q\abs{b-\ave{b}_Q}^2d\mu\Big)^{1/2}
\end{equation*}
is the dyadic BMO norm based on $L_2$ averages. (The different BMO norms are equivalent by the John--Nirenberg inequality, which remains valid in this generality.)

If we choose some $\lambda\gg\Norm{b}{\operatorname{BMO}^2(\mathscr D)}$, then $\mu(\bigcup S')\ll\mu(S)$. Denoting by $\mathscr S$ the collection of all stopping cubes of both kinds, and
\begin{equation*}
  E_S:=S\setminus\bigcup_{\substack{S'\in\mathscr S \\ S'\subsetneq S}}S',
\end{equation*}
these subsets $E_S\subseteq S$ are pairwise disjoint, and
\begin{equation*}
  \mu(E_S)\gtrsim\mu(S).
\end{equation*}
The existence of such subsets is referred to as the {\em sparseness} of the collection $\mathscr S$.

Let us now return to the estimation of $\Pi_b^{\sigma,\tau} f$. For each $Q\in\mathscr D_i$ for some $i\in[\sigma,\tau]$, let $\pi Q\in\mathscr S$ denote the minimal stopping cube that contains $Q$. Regrouping the summation under these stopping parents, we have
\begin{equation*}
\begin{split}
  \sum_{i=\sigma}^{\tau}\sum_{Q\in\mathscr D_i}\abs{\D_Q b}^2\ave{\Norm{f}{\mathcal{X}}}_Q^2
  &\leq\sum_{S\in\mathscr S}\sum_{\substack{ Q\in\mathscr D \\ \pi Q=S}}\abs{\D_Q b}^2\ave{\Norm{f}{\mathcal{X}}}_Q^2 \\
  &\lesssim\sum_{S\in\mathscr S}\Big(\sum_{\substack{ Q\in\mathscr D \\ \pi Q=S}}\abs{\D_Q b}^2\Big)\ave{\Norm{f}{\mathcal{X}}}_S^2 \\
  &\lesssim\sum_{S\in\mathscr S}\Norm{b}{\operatorname{BMO}^2(\mathscr D)}^2 1_S \ave{\Norm{f}{\mathcal{X}}}_S^2.
\end{split}
\end{equation*}
Hence
\begin{equation*}
\begin{split}
  &\BNorm{\Big(\sum_{i=\sigma}^{\tau}\sum_{Q\in\mathscr D_i}\abs{\D_Q b}^2\ave{\Norm{f}{\mathcal{X}}}_Q^2\Big)^{1/2}}{L_s(\mu)} \\
  &\qquad\lesssim \Norm{b}{\operatorname{BMO}^2(\mathscr D)}
  \BNorm{\Big(\sum_{S\in\mathscr S} 1_S \ave{\Norm{f}{\mathcal{X}}}_S^2\Big)^{1/2}}{L_s(\mu)} \\
\end{split}
\end{equation*}
Here 
\begin{equation*}
   \BNorm{\Big(\sum_{S\in\mathscr S} 1_S \ave{\Norm{f}{\mathcal{X}}}_S^2\Big)^{1/2}}{L_s(\mu)}
   \lesssim \BNorm{\sum_{S\in\mathscr S} 1_S \ave{\Norm{f}{\mathcal{X}}}_S}{L_s(\mu)}.
\end{equation*}
Dualising with $h\in L_{s'}(\mu)$ and using sparseness and the boundedness of the dyadic maximal operator
\begin{equation*}
  M_{\mathscr D}\phi(u):=\sup_{Q\in\mathscr D}1_Q(u)\ave{\abs{\phi}}_Q,
\end{equation*}
we find that
\begin{equation*}
\begin{split}
  \int_E\Big(\sum_{S\in\mathscr S} 1_S \ave{\Norm{f}{\mathcal{X}}}_S\Big)hd\mu
  &=\sum_{S\in\mathscr S} \ave{\Norm{f}{\mathcal{X}}}_S\ave{h}_S\mu(S) 
  \lesssim\sum_{S\in\mathscr S} \ave{\Norm{f}{\mathcal{X}}}_S\ave{h}_S\mu(E_S) \\
  &\leq\sum_{S\in\mathscr S}\int_{E_S} M_{\mathscr D}(\Norm{f}{\mathcal{X}})(u)M_{\mathscr D}h(u)d\mu(u) \\
  &\leq\int_E M_{\mathscr D}(\Norm{f}{\mathcal{X}})(u)M_{\mathscr D}h(u)d\mu(u) \\
  &\leq\bNorm{M_{\mathscr D}(\Norm{f}{\mathcal{X}})}{L_s(\mu)}\Norm{M_{\mathscr D}h}{L_{s'}(\mu)} \\
  &\lesssim\bNorm{\Norm{f}{\mathcal{X}}}{L_s(\mu)}\Norm{h}{L_{s'}(\mu)}=\Norm{f}{L_s(\mu;\mathcal{X})}\Norm{h}{L_{s'}(\mu)}.
\end{split}
\end{equation*}
Thus
\begin{equation*}
  \BNorm{\sum_{S\in\mathscr S} 1_S \ave{\Norm{f}{\mathcal{X}}}_S}{L_s(\mu)}
  \lesssim\Norm{f}{L_s(\mu;\mathcal{X})},
\end{equation*}
hence
\begin{equation*}
  \BNorm{\Big(\sum_{i=\sigma}^{\tau}\sum_{Q\in\mathscr D_i}\abs{\D_Q b}^2\ave{\Norm{f}{\mathcal{X}}}_Q^2\Big)^{1/2}}{L_s(\mu)} 
  \lesssim \Norm{b}{\operatorname{BMO}^2(\mathscr D)}\Norm{f}{L_s(\mu;\mathcal{X})}
\end{equation*}
and therefore
\begin{equation*}
  \Norm{\Pi_b^{\sigma,\tau}f}{L_s(\mu;\mathcal{X})}\lesssim n^{1/p-1/2}\Norm{b}{\operatorname{BMO}^2(\mathscr D)}\Norm{f}{L_s(\mu;\mathcal{X})},
\end{equation*}
provided that $\mathcal{X}$ has martingale type $p\in[1,2]$.

Recall that $b=T1$. That this belongs to BMO is a well known necessary condition for the boundedness of $T$ on $L_s(\mu)$. The standard argument becomes particularly simple in the present finite case, since $T1$ is a pointwise well-defined function (not an equivalence class modulo constants as in the general theory). We have
\begin{equation*}
  T1=T(1_{cB_Q})+T(1_{(cB_Q)^c})
\end{equation*}
and hence
\begin{equation*}
\begin{split}
  \Big(\int_Q &\abs{T1-\ave{T1}_Q}^sd\mu\Big)^{1/s}
  \leq 2\inf_z\Big(\int_Q\abs{T1-z}^sd\mu\Big)^{1/s} \\
  &\leq 2\Big(\int_Q\abs{T1-T(1_{(cB_Q)^c})(z_Q)}^sd\mu\Big)^{1/s} \\
  &\leq 2\Big(\int_Q\abs{T(1_{cB_Q})}^sd\mu\Big)^{1/s}+2\Big(\int_Q\abs{T(1_{(cB_Q)^c})-T(1_{(cB_Q)^c})(z_Q)}^sd\mu\Big)^{1/s}.
\end{split}
\end{equation*}
The first term is dominated by the ball testing condition
\begin{equation*}
  \Norm{T}{\operatorname{test}^s(\mathscr B)}:=\sup_{B\in\mathscr B}\frac{\Norm{T 1_B}{L_s(\mu)}}{\mu(B)^{1/s}},
\end{equation*}
where $\mathscr B$ is the family of all balls $B(u,t)\subset E$; namely
\begin{equation*}
  \Big(\int_Q\abs{T(1_{cB_Q})}^sd\mu\Big)^{1/s}
  \leq\Norm{T}{\operatorname{test}^s(\mathscr B)}\mu(cB_Q)^{1/s}\lesssim\Norm{T}{\operatorname{test}^s(\mathscr B)}\mu(Q)^{1/s}.
\end{equation*}
For the remaining integral, we observe that
\begin{equation*}
\begin{split}
  \int_Q &\abs{T(1_{(cB_Q)^c})(u)-T(1_{(cB_Q)^c})(z_Q)}^sd\mu(u) \\
  &\leq\int_Q\int_{(cB_Q)^c}\abs{K(u,v)-K(z_Q,v)}d\mu(v)d\mu(u) \\
  &\lesssim\int_Q\int_{(cB_Q)^c}\omega\Big(\frac{d(u,z_Q)}{d(v,z_Q)}\Big)\frac{1}{V(v,z_Q)}d\mu(v)d\mu(u).
\end{split}
\end{equation*}
The inner integral is bounded by a constant by computations like those in the proof of Lemma \ref{lem:tail}, and the outer integral over $Q$ then gives simply $\mu(Q)$.

\section{{Reorganisation of the remaining parts}}

After the extraction of the paraproducts, we are left with estimating sums of the form
\begin{equation}\label{eq:TPQ}
  \sum_{i=a}^b\sum_{P,Q\in\mathscr D_i}T(P,Q),\quad
  T(P,Q)\in\begin{cases}\pair{T\D_P f}{\D_Q g}, \\ (\ave{f}_P-\ave{f}_Q)\pair{T1_P}{\D_Q g}, \\ \pair{T\D_P f}{1_Q}(\ave{g}_Q-\ave{g}_P),\end{cases}
\end{equation}
where we note that the last two ``products'' are actually duality pairings between e.g. $(\ave{f}_P-\ave{f}_Q)\in \mathcal{X}$ and $\pair{T1_P}{\D_Q g}\in \mathcal{X}^\prime$.

Dropping for the moment the summands, and recalling the parameter $m_0$ guaranteed by Lemma \ref{lem:m0}, we further reorganise the summation as
\begin{equation*}
  \sum_{P,Q\in\mathscr D_i}
  =\sum_{\substack{P,Q\in\mathscr D_i \\ d(x_P,x_Q)\leq \frac12\eps\delta^{-m_0} \ell(P)}}
  +\sum_{m=m_0+1}^\infty\sum_{\substack{P,Q\in\mathscr D_i \\ \frac12\eps\delta^{1-m} \ell(P)<d(x_P,x_Q)\leq\frac12\eps\delta^{-m} \ell(P)}}.
\end{equation*}
In the $m$-th sum, the Haar coefficients (appearing when expanding the summands $T(P,Q)$) can be estimated by
\begin{equation}\label{eq:TPQest}
  \abs{\pair{Th_P^\alpha}{h_Q^\beta}}\lesssim\omega(\delta^{m})\frac{\sqrt{\mu(P)\mu(Q)}}{V(z_P,\delta^{-m}\ell(P))},
\end{equation}
which depends on the lower bound on $d(x_P,x_Q)$. The case $m=m_0$ is included, in which case this bound reduces to just $O(1)$, regarding $m_0$ as fixed.

Let us observe that we can also truncate the summation over $m$ from above. Due to the property that $K(u,v)\neq 0$ only if $d(u,v)<R$, a pairing like $\pair{T\D_P f}{\D_Q g}$ (or one with $\D_P f$ replaced by $1_P$, or $\D_Q g$ replaced by $1_Q$) can only be non-zero if $d(P,Q)<R$ and hence also $d(x_P,x_Q)\lesssim R$. But in the $m$-th term we also have the lower bound $d(x_P,x_Q)>\frac12\eps\delta^{1-m}\ell(P)\gtrsim\frac12\eps\delta^{1-m}r$. Hence this term gives no contribution unless $\delta^{-m}\lesssim R/r$, thus $m\lesssim n$.

To proceed further, we will need to make use of a random selection of our dyadic systems $\mathscr D$. We can write the decomposition of the pairing $\pair{Tf}{g}$ relative to any such dyadic system, and then take the expectation over the choice of $\mathscr D$. By linearity, we can move the expectation inside the summations. If we assume that $f$ and $g$ are boundedly supported (such functions are in any case dense in the spaces that we consider), we entirely avoid any issues of convergence: The summation over $i$ is finite already, and for each $i$, the bounded support of $f$ can only intersect finitely many of the cube $P,Q\in\mathscr D_i$. 

Let $A_m(P,Q)$ denote the random event---as in Lemma \ref{lem:m0}, only with a slightly different notation---that the two (random) cubes $P,Q\in\mathscr D_i$ share a common  ancestor in the (random) dyadic system $\mathscr D_{i-m}$. By the summation condition that $d(x_P,x_Q)\leq\frac12\eps\delta^{-m} \ell(P)$, Lemma \ref{lem:m0} guarantees that $\P(A_m(P,Q))\in[\frac12,1]$. Moreover, Lemma \ref{lem:m0} also says that $A_m(P,Q)$ depends only on $(\omega_j)_{i-m\leq j<i}$ (observe the difference in indexing compared to Lemma \ref{lem:m0}). On the other hand, the cubes $P,Q$, as well as their dyadic children, and hence quantities like $\D_P f$ and eventually $T(P,Q)$, depend only on $(\omega_j)_{j\geq i}$. Thus
\begin{equation*}%\label{eq:indep}
  T(P,Q)\text{ and }A_m(P,Q)\text{ are independent.}
\end{equation*}
We can then manipulate
\begin{equation*}
\begin{split}
  \E &\sum_{\substack{P,Q\in\mathscr D_i \\ \frac12\eps\delta^{1-m} \ell(P) <d(x_P,x_Q) \\ \leq\frac12\eps\delta^{-m} \ell(P)}}T(P,Q) \\
   &=\sum_{\substack{P,Q\in\mathscr D_i \\ \frac12\eps\delta^{1-m} \ell(P) <d(x_P,x_Q) \\ \leq\frac12\eps\delta^{-m} \ell(P)}}\E T(P,Q)\frac{\E 1_{A_m(P,Q)}}{\P(A_m(P,Q))} \\
   &=\sum_{\substack{P,Q\in\mathscr D_i \\ \frac12\eps\delta^{1-m} \ell(P) <d(x_P,x_Q) \\ \leq\frac12\eps\delta^{-m} \ell(P)}}\E[T(P,Q) 1_{A_m(P,Q)}]\frac{1}{\P(A_m(P,Q))} \\
   &=\E\sum_{\substack{P,Q\in\mathscr D_i \\ \frac12\eps\delta^{1-m} \ell(P) <d(x_P,x_Q) \\ \leq\frac12\eps\delta^{-m} \ell(P)}}T(P,Q) 1_{A_m(P,Q)}\frac{1}{\P(A_m(P,Q))} \\
   &=:\E\sum_{\substack{P,Q\in\mathscr D_i \\ \frac12\eps\delta^{1-m} \ell(P) <d(x_P,x_Q) \\ \leq\frac12\eps\delta^{-m} \ell(P)}} T_m(P,Q).
\end{split}
\end{equation*}
Thus, at the small cost of multiplying the summands by numerical factors in the range $[1,2]$, we could insert the extra summation condition $A_m(P,Q)$ that $P,Q\in\mathscr D_i$ share a common ancestor $S\in\mathscr D_{i-m}$.

We can now reorganise the sums under these common ancestors. Thus, abbreviating the summation condition
\begin{equation*}
  \begin{cases}
     d(x_P,x_Q)\leq \frac12\eps\delta^{-m_0}, & \text{if }m=m_0, \\
    \frac12\eps\delta^{1-m}< d(x_P,x_Q)\leq \frac12\eps\delta^{-m}, & \text{if }m>m_0,
  \end{cases}
\end{equation*}
simply as $\sum^m$, we have
\begin{equation*}
\begin{split}
  \E\sum_{i=\sigma}^\tau\sum_{P,Q\in\mathscr D_i}T(P,Q)
 % &=\E\sum_{m=m_0}^\infty\sum_{i=\sigma}^\tau \sum_{P,Q\in\mathscr D_i}^m T(P,Q) \\
  &=\E\sum_{m=m_0}^\infty\sum_{i=\sigma}^\tau \sum_{P,Q\in\mathscr D_i}^m \tilde T_m(P,Q) \\
  &=\E\sum_{m=m_0}^\infty\sum_{i=\sigma}^\tau \sum_{S\in\mathscr D_{i-m}}\sum_{\substack{P,Q\in\mathscr D_i \\ P,Q\subseteq S}}^m \tilde T_m(P,Q).
\end{split}
\end{equation*}
Let us investigate
\begin{equation*}
  A^m_S(f,g):=\sum_{\substack{P,Q\in\mathscr D_i \\ P,Q\subseteq S}}^m \tilde T_m(P,Q),
\end{equation*}
where $T_m(P,Q)$ is just $T(P,Q)$ multiplied by a number in $\{0\}\cup[1,2]$, and $T(P,Q)$ is as in \eqref{eq:TPQ}.

\subsection{Cancellation properties}
We see that both $f$ and $g$ are acted on by operators that annihilate constants, only observe these functions on $P,Q\subseteq S$. Thus, we may replace $f$ by $1_S(f-\ave{f}_S)$, and similarly with $g$. Moreover, in the case of $\D_P f$, we may replace $f$ by
\begin{equation*}
  \D^m_S f:=\sum_{\substack{S'\in\mathscr D \\ \ell(S')=\delta^m\ell(S)}}\D_{S'}f,
\end{equation*}
and in the case of $\ave{f}_P-\ave{f}_Q$, by
\begin{equation*}
  \E^m_S f:=\sum_{\substack{S'\in\mathscr D \\ \ell(S')=\delta^m\ell(S)}}\D_{S'}f,
\end{equation*}
as this operator only observed the averages of $f$ on the of the cubes $P$ and $Q$. Combining these observations, we can actually replace $f$, in the latter case, by
\begin{equation*}
  1_S(\E^m_S f-\ave{\E^m_S f}_S)=\sum_{\substack{S'\in\mathscr D \\ \ell(S')>\delta^m\ell(S)}}\D_{S'}f=:\D^{[0,m)}_S f.
\end{equation*}
Similar observations apply on the $g$ side. This describes the cancellation properties present in $A^m_S(f,g)$.

\subsection{Size properties}
Noting that $\ell(S)=\delta^{-m}\ell(P)$ and $z_P\in S$, we can simplify the denominator in \eqref{eq:TPQest} as
\begin{equation*}
  V(z_P,\delta^{-m}\ell(P))=V(z_P,\ell(S))\asymp \mu(S).
\end{equation*}
In the case of $\D_P$ and $\D_Q$ on both sides, we have
\begin{equation*}
  A^m_S(f,g)=\iint_S a^m_S(u,v)f(v)g(u)d\mu(v)d\mu(u),
\end{equation*}
where
\begin{equation*}
  a^m_S(u,v)=\sum_{\alpha,\beta}\sum_{\substack{P,Q\in\mathscr D_i \\ P,Q\subseteq S}}^m \chi^m_{\alpha,\beta}\pair{Th^\alpha_P}{h^\beta_Q}h^\alpha_P(v)h^\beta_Q(u)
\end{equation*}
satisfies
\begin{equation*}
\begin{split}
  \abs{a^m_S(u,v)}
  &\lesssim\sum_{\substack{P,Q\in\mathscr D_i \\ P,Q\subseteq S}}^m
  \omega(\delta^m)\frac{\sqrt{\mu(P)\mu(Q)}}{\mu(S)}\frac{1_P(v)}{\mu(P)^{1/2}}\frac{1_Q(u)}{\mu(Q)^{1/2}} \\
  &=\omega(\delta^m)\frac{1}{\mu(S)}\sum_{\substack{P,Q\in\mathscr D_i \\ P,Q\subseteq S}}^m 1_P(v)1_Q(u)
  \leq\omega(\delta^m)\frac{1}{\mu(S)} 1_{S\times S}(u,v).
\end{split}
\end{equation*}
In the case of $\ave{f}_P-\ave{f}_Q$, we get
\begin{equation*}
  a^m_S(u,v)=\sum_{\beta}\sum_{\substack{P,Q\in\mathscr D_i \\ P,Q\subseteq S}}^m \chi^m_{\alpha,\beta}\pair{T1_P}{h^\beta_Q}
  \Big(\frac{1_P(v)}{\mu(P)}-\frac{1_Q(v)}{\mu(Q)}\Big)h^\beta_Q(u),
\end{equation*}
which satisfies
\begin{equation*}
\begin{split}
  \abs{a^m_S(u,v)} &\lesssim
  \sum_{\substack{P,Q\in\mathscr D_i \\ P,Q\subseteq S}}^m 
  \omega(\delta^m)\frac{\mu(P)\sqrt{\mu(Q)}}{\mu(S)}\Big(\frac{1_P(v)}{\mu(P)}+\frac{1_Q(v)}{\mu(Q)}\Big)\frac{1_Q(u)}{\mu(Q)^{1/2}} \\
  &=\omega(\delta^m)\frac{1}{\mu(S)}\sum_{\substack{P,Q\in\mathscr D_i \\ P,Q\subseteq S}}^m
  \Big(1_P(v)1_Q(u)+\mu(P)\frac{1_Q(v)1_Q(u)}{\mu(Q)}\Big) \\
  &\leq\omega(\delta^m)\Big(\frac{1}{\mu(S)}1_{S\times S}(u,v)+\sum_{\substack{Q\subseteq S \\ \ell(Q)=\delta^m\ell(S)}}\frac{1_{Q\times Q}(u,v)}{\mu(Q)}\Big)
\end{split}
\end{equation*}
The case of $\ave{g}_Q-\ave{g}_P$ is of course symmetric.

Note that $1_{Q\times Q}(u,v)/\mu(Q)$ is the kernel of the averaging operator $\E_Q$.

Altogether, we find that $A^m_S$ takes one of the following forms:
\begin{equation*}
  A^m_S =\omega(\delta^m)\times\begin{cases} \D^m_S \dot A^m_S \D^m_S, \\ \D^{[0,m)}_S \dot A^m_S \D^m_S, \\ \D^m_S\dot A^m_S \D^{[0,m)}_S,
  \end{cases}
\end{equation*}
where
\begin{equation*}
  \Norm{\dot A^m_S f}{\mathcal{X}}\lesssim \E_S\Norm{f}{\mathcal{X}}+\E^m_S\Norm{f}{\mathcal{X}}.
\end{equation*}

\section{{Final estimates}}

We are left with estimating, e.g.,
\begin{equation}\label{eq:final0}
  \sum_{m=m_0}^\infty\omega(\delta^m) \sum_{i=\sigma}^\tau \pair{\D^{[0,m)}_{i-m}\dot A^m_{i-m} \D^m_{i-m} f}{g},
\end{equation}
where
\begin{equation*}
  G_{i}=\sum_{S\in\mathscr D_i}G_S,\quad G\in\{D^{[0,m)},\dot A^m,\D^m\},
\end{equation*}
as well as two similar sums with either the positions of $\D^{[0,m)}_S$ and $\D^m_S$ interchanged, or both replaced by $\D^m_S$.

Let us estimate the term that we wrote down above.

\begin{lemma}\label{lem:final1}
Let $1<p\leq s\leq q<\infty$ and suppose that $\mathcal{X}$ has martingale type $p$ and martingale cotype $q$. Then
\begin{equation*}
   \Babs{\sum_{i=\sigma}^\tau \pair{\D^{[0,m)}_{i-m}\dot A^m_{i-m} \D^m_{i-m} f}{g}}
   \lesssim n^{1/p-1/q}m^{1/p'}\Norm{f}{L_s(\mu;\mathcal{X})}\Norm{g}{L_s(\mu;\mathcal{X}^\prime)}.
\end{equation*}
\end{lemma}

\begin{proof}
By self-adjointness of the difference operators,
\begin{equation*}
\begin{split}
  \abs{\pair{\D^{[0,m)}_i\dot A^m_i \D^m_i f}{g}}
  &=\abs{\pair{\dot A^m_i \D^m_i f}{\D^{[0,m)}_i g}} \\
  &\leq\pair{(\E_i+\E_{i+m})\Norm{\D^m_i f}{\mathcal{X}}}{\Norm{\D^{[0,m)}_ig}{\mathcal{X}^\prime}}.
\end{split}
\end{equation*}
Hence
\begin{equation*}%\label{eq:final1}
\begin{split}
  \sum_{i=\sigma}^\tau & \abs{\pair{\D^{[0,m)}_{i-m}\dot A^m_{i-m} \D^m_{i-m} f}{g}} 
  =\sum_{i=\sigma-m}^{\tau-m} \abs{\pair{\D^{[0,m)}_{i}\dot A^m_{i} \D^m_{i} f}{g}}   \\
  &\leq\BNorm{\Big(\sum_{i=\sigma-m}^{\tau-m} \big[(\E_i+\E_{i+m})\Norm{\D^m_i f}{\mathcal{X}}]^q\Big)^{1/q}}{L_s(\mu)} \\
  &\qquad\times  \BNorm{\Big(\sum_{i=\sigma-m}^{\tau-m} \big[\Norm{\D^{[0,m)}_i g}{\mathcal{X}}]^{q'}\Big)^{1/q'}}{L_{s'}(\mu)}
\end{split}
\end{equation*}
for any $q,s\in(1,\infty)$. In the first term, we can pull out the conditional expectations (even after a crude estimate $(\E_i+\E_{i+m})\phi\leq\sup_{j\in\Z}\E_j\phi$ if we wish) by the vector-valued Doob maximal inequality \cite[Theorem 3.2.7]{HNVW1}:
\begin{equation*}
  \BNorm{\Big(\sum_{i=\sigma-m}^{\tau-m} \big[(\E_i+\E_{i+m})\Norm{\D^m_i f}{\mathcal{X}}]^q\Big)^{1/q}}{L_s(\mu)}
  \lesssim  \BNorm{\Big(\sum_{i=\sigma-m}^{\tau-m} \Norm{\D^m_i f}{\mathcal{X}}^q\Big)^{1/q}}{L_s(\mu)}.
\end{equation*}
If $\mathcal{X}$ has martingale cotype $q$, then
\begin{equation*}
   \BNorm{\Big(\sum_{i=\sigma-m}^{\tau-m} \Norm{\D^m_i f}{\mathcal{X}}^q\Big)^{1/q}}{L_s(\mu)}
   \lesssim\Norm{f}{L_s(\mu;\mathcal{X})},
\end{equation*}
essentially by definition, since $\D^m_i f$ are martingale differences.

The dual side is slightly trickier, since the compound differences $$\D^{[0,m)}_i=\sum_{k=0}^{m-1}\D_{i-k}$$ involve overlapping scales. For the optimal estimate, we need to be somewhat clever with the use of the triangle inequality. This is accomplished as follows: We now assume that $p\leq s\leq q$; hence $q'\leq s'\leq p'$, where $\mathcal{X}$ has martingale type $p$ and martingale cotype $q$. Then
\begin{equation*}
\begin{split}
  &\BNorm{\Big(\sum_{i=\sigma-m}^{\tau-m} \Norm{\D^{[0,m)}_i g}{\mathcal{X}}^{q'}\Big)^{1/q'}}{L_{s'}(\mu)} \\
  &\qquad\leq n^{1/q'-1/s'}\BNorm{\Big(\sum_{i=\sigma-m}^{\tau-m} \Norm{\D^{[0,m)}_i g}{\mathcal{X}^\prime}^{s'}\Big)^{1/s'}}{L_{s'}(\mu)} \\
  &\qquad=n^{1/q'-1/s'}\Big( \sum_{j=0}^{m-1}\BNorm{\Big(\sum_{\substack{i=\sigma-m \\ i\equiv j\mod m}}^{\tau-m} \Norm{\D^{[0,m)}_i g}{\mathcal{X}^\prime}^{s'}\Big)^{1/s'}}{L_{s'}(\mu)}^{s'}\Big)^{1/s'}.
\end{split}
\end{equation*}
The inner summation contains every $m$-th term among a sequence of length $n$. Recalling that $m\lesssim n$, the number of such terms is $O(n/m)$. (There could of course be a single term even if $m\gg n$, in which case the estimate would no longer be valid.) Thus, by H\"older's inequality, we can continue with
\begin{equation*}
  \lesssim  n^{1/q'-1/s'}(n/m)^{1/s'-1/p'}
  \Big( \sum_{j=0}^{m-1}\BNorm{\Big(\sum_{\substack{i=\sigma-m \\ i\equiv j\mod m}}^{\tau-m} \Norm{\D^{[0,m)}_i g}{\mathcal{X}^\prime}^{p'}\Big)^{1/p'}}{L_{s'}(\mu)}^{s'}\Big)^{1/s'}.
\end{equation*}
For each fixed $j$, the sequence $(\D^{[0,m)}_i )_{i\equiv j\mod m}$ consists of martingale differences, If $\mathcal{X}$ has martingale type $p$, its dual $\mathcal{X}^\prime$ has martingale cotype $p'$, and we can further continue the estimate
\begin{equation*}
\begin{split}
  &\lesssim n^{1/q'-1/s'}(n/m)^{1/s'-1/p'}
  \Big( \sum_{j=0}^{m-1}\Norm{g}{L_{s'}(\mu;\mathcal{X}^\prime)}^{s'}\Big)^{1/s'} \\
  &=n^{1/q'-1/s'}(n/m)^{1/s'-1/p'}(n/m)^{1/s'-1/p'}m^{1/s'}\Norm{g}{L_{s'}(\mu;\mathcal{X}^\prime)} \\
  &=n^{1/q'-1/p'}m^{1/p'}\Norm{g}{L_{s'}(\mu;\mathcal{X}^\prime)}=n^{1/p-1/q}m^{1/p'}\Norm{g}{L_{s'}(\mu;\mathcal{X}^\prime)}.
\end{split}
\end{equation*}
\end{proof}

By symmetry, we infer that
\begin{equation*}
     \Babs{\sum_{i=\sigma}^\tau \pair{\D^m_{i-m}\dot A^m_{i-m} \D^{[0,m)}_{i-m} f}{g}}
   \lesssim n^{1/p-1/q}m^{1/q}\Norm{f}{L_s(\mu;\mathcal{X})}\Norm{g}{L_s(\mu;\mathcal{X}^\prime)}.
\end{equation*}

In the case of $\D^m_{i-m}$ on both sides, the estimation is more straightforward, and we simply estimate
\begin{equation*}
\begin{split}
  &\BNorm{\Big(\sum_{i=\sigma-m}^{\tau-m} \Norm{\D^{m}_i g}{\mathcal{X}}^{q'}\Big)^{1/q'}}{L_{s'}(\mu)} \\
  &\qquad\leq n^{1/q'-1/p'}\BNorm{\Big(\sum_{i=\sigma-m}^{\tau-m} \Norm{\D^{,}_i g}{\mathcal{X}^\prime}^{p'}\Big)^{1/p'}}{L_{s'}(\mu)}
  \lesssim n^{1/q'-1/p'}\Norm{g}{L_{s'}(\mu;\mathcal{X}^\prime)}
\end{split}
\end{equation*}
and hence
\begin{equation*}
   \Babs{\sum_{i=\sigma}^\tau \pair{\D^m_{i-m}\dot A^m_{i-m} \D^{m}_{i-m} f}{g}}
   \lesssim n^{1/p-1/q}\Norm{f}{L_s(\mu;\mathcal{X})}\Norm{g}{L_s(\mu;\mathcal{X}^\prime)}.
\end{equation*}
The largest of the different bounds that we have obtained for the $m$-th term is hence
\begin{equation*}
  n^{1/p-1/q}m^{\max(1/p',1/q)}\Norm{f}{L_s(\mu;\mathcal{X})}\Norm{g}{L_s(\mu;\mathcal{X}^\prime)}.
\end{equation*}
Recalling from \eqref{eq:final0} that we still need to sum over $m$, we finally obtain the bound
\begin{equation*}
\begin{split}
  n^{1/p-1/q}\sum_{m=m_0}^\infty \omega(\delta^m)m^{\max(1/p,1/q')}
  &\asymp n^{1/p-1/q}\int_0^1 \omega(t)(1+\log\frac1t)^{\max(1/p,1/q')}\frac{dt}{t} \\
  &=:n^{1/p-1/q}\Norm{\omega}{\operatorname{Dini}^{\max(1/p,1/q')}}.
\end{split}
\end{equation*}

\section{{Concluding remarks}}

We have now concluded the (unavoidably somewhat lengthy) proof of Theorem \ref{thm:main}. In the course of the proof, after some preliminary reductions, the main part of the operator was decomposed into two paraproducts and the cancellative part. Recall that we assume that our target Banach space has martingale type $p$ and martingale cotype $q$.

For the paraproduct, we found the estimate $O(n^{1/p-1/2})$. By symmetry, the dual paraproduct then satisfies the bound $O(n^{1/q'-1/2})=O(n^{1/2-1/q})$. Finally, the cancellative part of the operator had the bound $O(n^{1/p-1/q})$, provided that the kernel satisfies a Dini condition of order $\max(1/p,1/q')$.

It might at first seem counterintuitive that the order of the required Dini condition increases with improving martingale type or cotype. On the other hand, the dependence on the finiteness parameter $n$ of the kernel decreases at the same time. A way to think of this is as follows: If the space is better, then we are able to make better use of the properties of the kernel, whereas a poor space only observes a Dini condition of low order, even if more would be available.

If (and only if) $\mathcal{X}$ is (isomorphic to) a Hilbert space, then one can take $p=q=2$, and all bounds become $O(1)$ in terms of $n$, allowing one to dispense with the truncations and deal with genuine singular integrals. In this case, one needs a Dini condition of order $1/2$ to run the argument. It seems to be open whether this can be relaxed, even in the scalar case.

In fact, with the present argument, the only way to get $O(1)$ in terms of $n$ in the cancellative term is to have $p=q=2$. This is due to the fact that necessarily $p\leq 2$ and $q\geq 2$; hence $1/p-1/q=(1/p-1/2)+(1/2-1/q)$ can only be zero if both terms vanish. On the other hand, for the paraproducts, it suffices to have just one of $p$ or $q$ equal to $2$, which provides a much larger class of examples.

The boundedness of the dyadic paraproduct $\Pi_b$ with a scalar-valued symbol $b\in\operatorname{BMO}$, is well known on $L_s(\mu;\mathcal{X})$, in the case that $\mathcal{X}$ is a UMD space. The result is attributed to Bourgain, and written down in \cite{FigWoj}. Our considerations in Section \ref{sec:parap} show that they are also bounded on $L_s(\mu;\mathcal{X})$, if $\mathcal{X}$ has martingale type $2$. These two classes (UMD and martingale type $2$) are not mutually comparable. This raises the interesting question about the maximal class of spaces $\mathcal{X}$ such that the dyadic paraproduct induces a bounded operator on $L_s(\mu;\mathcal{X})$.

{\bf Acknowledgments.} I would like to thank the anonymous referees for their constructive comments on the manuscript.

%\bibliography{t1}
%\bibliographystyle{abbrv}

% ------------------------------------------------------------------------
\end{document}